\documentclass[11pt]{amsart}
\usepackage{cases}
\usepackage{amsfonts}
\usepackage{graphicx}
\usepackage{amssymb}
\usepackage{mathrsfs}
\usepackage{ulem}
\usepackage{xcolor}
\usepackage{hyperref}
\usepackage{enumerate}

\newtheorem{theorem}{Theorem}[section]
\newtheorem{corollary}[theorem]{Corollary}
\newtheorem{lemma}[theorem]{Lemma}

\theoremstyle{definition}
\newtheorem{definition}[theorem]{Definition}
\newtheorem{remark}[theorem]{Remark}

\numberwithin{equation}{section}

\newcommand{\R}{\mathbb{R}}

\newcommand{\Cn}{\mathbb{C}^n}
\newcommand{\C}{\mathbb{C}}

\newcommand{\fr}{\frac}

\DeclareMathOperator{\esssup}{ess\, sup}

\begin{document}

\title[Fredholm Toeplitz operators on doubling Fock spaces]{Fredholm Toeplitz operators on doubling Fock spaces}

\author{Zhangjian Hu}
\address{Department of Mathematics, Huzhou University, Huzhou, Zhejiang, China}
\email{huzj@zjhu.edu.cn}

\author{Jani A. Virtanen}
\address{Department of Mathematics and Statistics, University of Reading, England {\rm and} Department of Mathematics and Statistics, University of Helsinki, Finland}
\email{j.a.virtanen@reading.ac.uk {\rm and} jani.virtanen@helsinki.fi}

\keywords{Toeplitz operators, Fredholm properties, doubling Fock spaces, vanishing mean oscillation, quasi-Banach spaces}

\thanks{Hu was supported in part by   the
National Natural Science Foundation of China (12071130, 12171150) and Virtanen was supported in part by Engineering and Physical Sciences Research Council
grant EP/T008636/1.}

\subjclass[2010]{Primary 47B35; Secondary 30H20}

\begin{abstract}
Recently the authors characterized the Fredholmn properties of Toeplitz operators on weighted Fock spaces when the Laplacian of the weight function is bounded below and above. In the present work the authors extend their characterization to doubling Fock spaces with a subharmonic weight whose Laplacian is a doubling measure. The geometry induced by the Bergman metric for doubling Fock spaces is much more complicated than that of the Euclidean metric used in all the previous cases to study Fredholmness, which leads to considerably more involved calculations.
\end{abstract}

\maketitle

\section{Introduction}

Let  $\Omega$ be a   domain of the complex plane ${\C}$, and let $\mu$ be  a positive  Borel measure  on $\Omega$. For $0<p<\infty$, the space $L^p(\Omega, d\mu)$ consists of all Lebesgue measurable functions $f$ on $\Omega$ for which
$$
\|f\|_{L^p(\Omega, d\mu)}=\left(\int_{\Omega}|f(z)|^{p} d\mu(z)\right)^{\frac{1}{p}}<\infty.
$$
We denote by $H({\Omega})$ the family of all holomorphic functions on  $\Omega$.

A positive  Borel measure $\nu$ on ${\C}$ is called doubling if there exists some constant $M>0$ such that
$$
	\nu(D(z, 2r))\leq M\nu(D(z, r))
$$
for $z \in{\C}$ and $r>0$, where $D(z,r)=\left\{w\in{\C}: |w-z|<r \right\}$. The smallest value of $M$ in (1.1) is called the doubling constant of $\nu$.
We denote by $dA$ the Lebesgue area measure on ${\C}$. For a subharmonic function   $\phi$, not identically zero on $\C$,  with  $\nu= \Delta\phi\, dA$    doubling,  we denote by  $\rho(z)$   the positive radius such that $\nu\left(D(z, \rho(z)\right)=1$.
 The function $\rho^{-2}$ can be viewed  as a regularized version of $\Delta\phi $, see \cite{Ch91} and \cite{MMO03}.

For $0< p<\infty$, we write $L^p_{\phi}$ instead of $L^p({\C}, e^{-p\phi} dA)$ for simplicity, that is, we say $f\in L^p_{\phi}$ if $f$ is Lebesgue measurable on $\C$ and
$$
	\|f\|_{p, \phi}^p=\int_{{\C}}|f(z)|^{p}  e^{-p\phi(z)} dA(z)<\infty.
$$
By $L^{\infty}_{\phi}$ we denote the set of all Lebesgue measurable functions $f$ on ${\C}$ for which
$$
	\|f\|_{\infty, \phi}=\esssup\limits_{z\in {\C}}|f(z)|  e^{-\phi(z)} <\infty.
$$
The doubling Fock space $F^p_\phi$  is defined by
$$
 	F^p_{\phi}= L^p_{\phi}\cap H(\C).
$$
Under $\|\cdot\|_{p, \phi}$, $F^p_{\phi}$ is a Banach space when $1\le p\le \infty$  and it is a quasi-Banach  space when $0<p< 1$. It is worth noting that $F^p_\phi$ is of infinite dimension if $\Delta\phi\,dA$ is a nontrivial doubling measure (see, e.g., \cite{Ch91, MMO03, MO09}); for further details on the dimension of weighted Fock spaces, see~\cite{BLY2021}. In a larger framework, the doubling Fock spaces $F^2_\phi$ can be viewed as Bergman spaces with admissible weights in the sense of~\cite{P-W1990}, but this connection will not be exploited in our present work.

The Fock spaces considered in the present paper cover a great deal in the literature. In particular, when $\phi(z)=\frac{\alpha}{2}|z|^2$ with  $\alpha>0$,  $F^2_{\phi}$ is  the classical Fock space $F^p_\alpha$, which has been studied by many authors, see, e.g. \cite{ JPR87, Tu05, Zh12}  and the
references therein.   If $\phi(z)=m\log|z|+|z|^2$, $m$ is a positive integer, $F^2_{\phi}$ is just the Fock-Sobolev space discussed in \cite{CCK12} and \cite{CZ12}. For  $\phi(z)=|z|^m$, $F^2_{\phi}$ is  the weighted Fock space  studied in \cite{SS09} and \cite{SS11}. Doubling Fock spaces also include the so-called generalized Fock spaces $F^p_\phi$, where $0<m<\Delta\phi<M$ (see~\cite{HV19} and the references therein).

Let $K(\cdot, \cdot)$ be the reproducing kernel of the Hilbert space $F^2_{\phi}$.
 The asymptotic behavior of  $K(\cdot, \cdot)$ has  been  studied in  \cite{Ch91} and \cite{MO09} for example.
For $0< p<\infty$ and $z\in {\C}$, set  $k_{p, z}(w)= {K(w, z)}/{\|K(\cdot, z)\|_{p,\phi}} $ to be the normalized Bergman kernel for $F^{p}_{\phi}$. We simply write $k_{z}$ for $k_{2, z}$. It is easy to see that $ k_{z}(w)= {K(w,z)}/{\sqrt{K(z,z)}} $.  With $K(\cdot, \cdot)$, we define an integral operator $P$ (known as the Bergman projection) by
$$
	Pf(z)=\int_{{\C}}K(z,w)f(w)e^{-2\phi(w)}dA(w)
$$
for $z\in \C$. With this integral representation, given a Lebesgue  measurable function $f$ on $\C$, we define the Toeplitz operator $T_f$ and the Hankel operator $H_f$ on $F^p_\phi$, respectively, as
$$
	T_fg =P (fg)
$$
and
$$
	H_fg=  (I-P)(fg) = fg - P(fg)
$$
provided that the corresponding integral $P(fg)$ makes sense, where $I$ is the identity operator on functions. The function $f$ generating Toeplitz and Hankel operators is referred to as the symbol of the operator under consideration.

In the present paper, we investigate the properties of Toeplitz and Hankel operators acting on doubling Fock spaces with symbols of bounded and vanishing mean oscillation. We denote these symbol classes by $BMO^p$ and $VMO^p$, respectively---their precise definitions and basic properties, such as the decompositions $BMO^p = BO \cap BA^p$ and $VMO^p = VO + VA^p$, will be given in the next section.

Our main focus is on the study of Fredholmness of Toeplitz operators on doubling Fock spaces. A bounded operator $T$ on a vector space $X$ is said to be Fredholm if its kernel $\ker T$ and cokernel $X/T(X)$ are both finite-dimensional. Fredholm operators are important for a variety of reasons, such as their role in global analysis, spectral theory and numerical analysis, and in many other branches of mathematics and mathematical physics. The study of the Fredholm properties of Toeplitz operators on Hardy spaces with continuous symbols is the culmination of works of several authors in the 1950s, including Gohberg, Simonenko and Widom, depending on the precise setting of the underlying Hardy space $H^p$. Since then there has been a steady increase in the study of spectral properties of Toeplitz operators on Hardy spaces and Bergman spaces, and in particular the Fredholm properties of Toeplitz operators on these spaces are well understood for several classes of symbols. Regarding Toeplitz operators on Fock spaces, most efforts have been focused on the study of boundedness, compactness, and Schatten class properties, and until recently there were only two articles on their Fredholm properties (see~\cite{BC87, MR1137884}, which both deal with the classical Fock space $F^2$ and bounded symbols of vanishing mean oscillation). In the past three years, a series of works have appeared (see~\cite{AHV, AV18, FH17, HV19}), which has brought the level closer to that of the other two spaces. Indeed, the following characterization was obtained independently in~\cite{AV18} and~\cite{FH17} using different methods, which we discuss in some detail below.

\begin{theorem}\label{classical Fock}
Let $f\in L^\infty \cap VO$, $1<p<\infty$ and $\alpha>0$. Then $T_f$ is Fredholm on the classical Fock space $F^p_\alpha$ if and only if there are positive numbers $\epsilon$ and $\delta$ such that
\begin{equation}\label{e:classic}
	|f(z)| \geq \epsilon\quad {\rm whenever}\ |z|\geq \delta.
\end{equation}
\end{theorem}

\noindent This result naturally leads to the question of whether the preceding theorem remains true for
\begin{enumerate}[(i)]
\item unbounded symbols;

\item small exponents;

\item more general weights.
\end{enumerate}
Very recently, in~\cite{HV19}, we have managed to make significant progress in answering these questions and in particular found a condition that~\eqref{e:classic} should be replaced with as stated in the following theorem.

\begin{theorem} \label{generalized Fock}
Let $f\in VMO^1$ and $0<p<\infty$.  Then the Toeplitz operator $T_f$ is Fredholm on the generalized Fock space $F^p_\phi$ with $0<m<\Delta \phi<M$ if and only if
\begin{equation}\label{e:condition}
	    0<\liminf_{|z|\to \infty} |\widetilde{f}(z)|
	    \quad {\rm and}\quad
	    \limsup_{|z|\to \infty} |\widetilde{f}(z)|< \infty,
\end{equation}
 where $\tilde f$ is the Berezin transform of the symbol $f$.
 \end{theorem}

Notice that it is not difficult to see that Theorem~\ref{classical Fock} follows immediately from Theorem~\ref{generalized Fock}. To completely settle the theory of Fredholm Toeplitz operators $T_f$ with symbols in $VMO^1$ on {\it all} Fock spaces $F^p_\phi$ with $f\in VMO^1$, we prove that Theorem~\ref{generalized Fock} remains true even for weights whose Laplacian is a doubling measure. This is our main result---see Theorem~\ref{Fredholm}. Of course, there are Fock spaces $F^p_\phi$ that cannot be viewed as doubling Fock spaces and hence further work related to (iii) remains. However, before such study can proceed, we would first need to characterize bounded (and compact) Toeplitz operators on more general Fock spaces than those with doubling weights. Here the main obstacle is that the Bergman kernel estimates are only known up to the doubling weights. Another natural problem is the study of higher dimensional setting---we discuss this and other types of Fock spaces briefly in Section~\ref{open problems and conclusions}.

We should note that recently, in~\cite{AHV}, an attempt was made to address (iii) when $p=2$ and in particular to prove Theorem~\ref{classical Fock} in the Hilbert space setting $F^2_\phi$ with doubling weights $\phi$. However, while this generalization is true as seen from our main theorem, its proof contains a gap in the construction of a regularizer because Proposition~4.1 of~\cite{AHV} is only known to be true when the vanishing (mean) oscillation is defined with respect to the Euclidean metric as opposed to the Bergman metric---see Remark~\ref{remark on AHV} for further details. A correct argument to this problem is given in the proof of Theorem~\ref{Fredholm}. This further demonstrates the complexities and pitfalls one encounters when dealing with the geometry of doubling Fock spaces.

For further details about the history of the problem, see~\cite{AV18, FH17, HV19} and the references therein. We limit our discussion here to the methodologies that have been used to treat the problem previously.

The first approach for the Hilbert space $F^2_\alpha$ was based heavily on Hilbert space and $C^*$-algebra techniques, the use of the Weyl operator, the Berezin transform, and the Heisenberg group; see~\cite{BC87}. Aside from the Berezin transform, these are not suitable for operators acting on Banach spaces---whether one can employ these techniques for doubling Fock spaces $F^2_\phi$ may be an interesting question but we do not address it in our present work.

Recently developed limit operator techniques due to Hagger and his collaborators offer an efficient way of dealing with Toeplitz operators on Banach spaces that possess sufficiently ``nice bases'' (see~\cite{FH17} and the references therein). Briefly, the idea is to densely embed $\C^n$ into the maximal ideal space $\mathcal M$ of $BUC$ (the space of bounded uniformly continuous functions), and view $\mathcal M\setminus \C^n$ as the boundary of $\C^n$. Then the boundary values of an operator $T$ in the Toeplitz algebra $\mathcal{T}_{p,\alpha}$ at $\mathcal{M}\setminus \C^n$ are obtained by ``shifting'' $T$ to the boundary (for the precise meaning, see~\cite{FH17}). This way, for each $x\in \mathcal{M}\setminus \C^n$, one obtains the so-called limit operator $T_x$. Using limit operator techniques, the following result was proved in~\cite{FH17}.

\begin{theorem}
Let $1<p<\infty$. Then an operator $T$ in the Toeplitz algebra $\mathcal{T}_{p,\alpha}$ on $F^p_\alpha$ is Fredholm if and only if all of its limit operators are invertible.
\end{theorem}

As a corollary of the preceding result, Theorem~\ref{classical Fock} was obtained in~\cite{FH17}. Limit operator techniques were further developed in a more abstract setting in~\cite{HS19} and applied to Toeplitz operators acting on the Fock-Sobolev spaces. However, it is currently not known whether the limit operator techniques can be used to treat more general Fock (or Bergman) spaces that do not possess nice bases or explicit formulas for the reproducing kernels. In particular, one of the central objects is the weighted shift operator $C_z : L^p_\alpha \to L^p_\alpha$ defined for each $z\in\C^n$ by
$$
	C_zf(w) = f(w-z) e^{\alpha \langle w,z \rangle - \tfrac{\alpha}2|z|^2} \qquad(w\in\C^n).
$$
Now observe that the shift operators are connected to the Weyl operator, which cannot be defined in the more general setting of Fock spaces $F^2_\phi$, and hence the limit operator techniques do not seem to be suitable for more general weights without further adjustments. We should say that when they can be used, they offer additional benefits, such as the treatment of Toeplitz algebras and certain other more abstract characterizations of Fredholm properties. On the other hand, there appears no way of using limit operators to obtain the usual formula for the index of Fredholm Toeplitz operators on $F^p_\alpha$ (see, e.g., Theorem~20 of~\cite{AV18}).

This brings us to the approach of the present work, which can be described as function theoretic and a more direct way of dealing with the problem at hand. More precisely, we use upper pointwise estimates for the Bergman kernel of the doubling Fock space~\cite{MO09}, a certain auxiliary integral operator whose kernel is related to the reproducing kernel and the Bergman metric, and provide characterizations of bounded and compact Hankel operators, which may be of interest in their own right.

The use of the Bergman metric in the context of the doubling Fock spaces $F^2_\phi$ originates in the work on interpolation and sampling~\cite{MMO03} and no other metric has ever been used to deal with $F^2_\phi$. This is indeed natural and expected because the Bergman metric corresponds to the Bergman kernel of the doubling Fock space $F^2_\phi$ in the same way as the Euclidean metric corresponds to the Bergman kernel of the standard Fock space $F^2_\alpha$. The main difficulty in our work is caused by the geometry induced by the Bergman metric for doubling Fock spaces, which is much more complicated than that of the Euclidean metric used previously in the study of Fredholm properties, and in particular the role of the Euclidean disks $D(z,r)$ is played by the sets of the form $D^r(z) = D(z, r\rho(z))$, where $\rho$ is not constant in general. Finally, we note that the nesting property $F^p_\phi \subset F^q_\phi$ ($p\leq q$) known for weights whose Laplacian is bounded below and above is no longer true for doubling Fock spaces, which is yet another complication that we have to deal with.

\section{Preliminaries}\label{preliminaries}
In this section we recall and prove some key lemmas on the Bergman metric, the reproducing kernel and their related integral and norm estimates.

\textit{Notation.} Throughout, we use $C$ to denote positive constants
whose value may change from line to line but does not depend  on the
functions being considered. We say that two quantities $A$ and $B$ are equivalent, and write $A\simeq B$, if
there exists some positive number $C$ such that $C^{-1 }A\le B \le C A$.

Suppose that $d\nu=\Delta\phi\, dA$ is a doubling measure. Then, for all $z\in \C$ and $r>0$,
$$
	\mu\left(\partial(D(z,r)\right) = \mu(\{z\}) = 0
$$
(so that doubling measures have no mass on circles), and, since $\nu$ is a locally finite nonzero doubling measure on $\C$,
$$
	0<\mu\left(D(z,r)\right) < \infty.
$$
It follows that, for each $z\in\C$,
\begin{equation}\label{e:limit}
	\lim_{r\to\infty} \mu\left(D(z,r)\right) = \infty,
\end{equation}
and the function $r\mapsto \mu\left(D(z,r)\right)$ is an increasing homeomorphism from the interval $(0,\infty)$ onto itself. Thus, for every $z\in \C$, there is a unique positive radius $\rho(z)$ such that
\begin{equation*}
	\mu(D(z,\rho(z)) = 1.
\end{equation*}
We also note that~\eqref{e:limit} implies that $F^p_\phi$ is of infinite dimension.

For $r>0$ and $z\in {\mathbb C}$, we write $D^{r}(z)=D(z, r\rho(z))$,  and $D(z)=D^1(z)$ for short. By  \cite{MMO03},  there exist some absolute constants $\theta$ and $C>0$ such that,
for  $z\in {\mathbb C}$ and  $w\in D^{r}(z)$,
\begin{equation}\label{eq:1}
	\rho(w)\simeq\rho(z) \textrm{ if }r\leq 1,
	\textrm{  and  }\frac{1}{Cr^{\theta}}\leq\frac{\rho(w)}{\rho(z)}
	\leq Cr^{\theta}  \textrm{ if }r>1.
\end{equation}
 Then, for $r>0$, there exists some constant $\alpha>0$ such that
\begin{equation}\label{rho-equivalence}
\frac{1}{\alpha}\rho(z)\leq \rho(w)\leq \alpha\rho(z)
\end{equation}
for $z\in {\mathbb C}$ and  $w\in D^{r}(z)$. From \eqref{rho-equivalence} and the triangle inequality,   for    $r>0$, it follows that there are $m_1=m_1(r)$ and $m_2=m_2(r)$ so that
\begin{equation}\label{rho-equivalence-a}
	D^{r}(z)\subseteq D^{m_1r}(w)  \  \textrm{ and }  \  D^{r}(w)\subseteq D^{m_2r}(z)  \  \textrm{ whenever } w\in D^{r}(z).
\end{equation}
Clearly,  $m_j>1$ for $j=1, 2$,  and further, we have
$$
	\tau= \sup_{0<r\le 1} \left[m_1(r)+ m_2(r)\right]<\infty.
$$

Given a sequence $\{a_j\}_{j=1}^{\infty}\subset{\C}$ and $r>0$,  we call $\{a_j\}_{j=1}^{\infty}$  an $r$-lattice if $\left\{D^{r}(a_j)\right\}_{j=1}^{\infty}$ covers ${\C}$ and the disks of  $\left\{D^{\frac{r}{5}}(a_j)\right\}_{j=1}^{\infty}$ are pairwise disjoint. The existence of $r$-lattice follows from a standard covering lemma---see Theorem~2.1 of~\cite{Ma95} and also~\cite{MMO03}. In addition, if $\{a_j\}_{j=1}^{\infty}$ is an $r$-lattice,
for $m> 0$,  there   exists an integer $N$  such that,  for each $z\in {\C}$, $D^{mr}(z)$ can intersect at most $N$ discs  of  $\left\{D^{r}(a_j)\right\}$.

For $z,w \in {\C}$, the distance $d_\phi$ induced by the metric $\rho^{-2} dz\otimes d\overline z$ is given by
$$
	d_{\phi}(z, w) =\inf\limits_{\gamma}\int_{0}^{1}\frac{|\gamma'(t)|}{\rho(\gamma(t))}dt,
$$
where the infimum is taken over all piecewise $C^{1}$ curves $\gamma: [0,1]\rightarrow {\C}$ with $\gamma(0)=z$ and $\gamma(1)=w$.  It is known that the Bergman metric 
$$
	\fr {\partial^2 \log K(z, z)}{\partial z\partial {\overline z}}dz\otimes d\overline z,
$$
which is given by the solution to the extremal problem
$$
	\frac{\sup\{ |f'(x)| : f\in F^2_\phi, f(z) = 0, \|f\|_{2,\phi} = 1\}}{\sqrt{K(z,z)}},
$$
is comparable to the metric $\rho^{-2} dz\otimes d\overline z$ (see page 355 of \cite{CO11} for further details). We write $\beta(\cdot, \cdot)$ for the Bergman distance on $\C$.  Therefore,  there are two positive constants $C_1$ and $C_2$ such that
\begin{equation}\label{equi-Bergman-dist}
	C_1 d_\phi(z, w)\le \beta(z, w)\le C_2 d_\phi(z, w).
\end{equation}
for $z, w \in \C$. We refer to~\cite{Krantz1, Krantz2} for further details on the basic properties of the Bergman metric and distance.

We need the following estimates for the Bergman distance $\beta$, instead of $d_\phi$, because $\beta$ is more suitable for the analysis of operators on Fock spaces.

\begin{lemma}\label{distance-d}
There exists $\delta \in (0,1)$  such that for every $r>0$ there exists $C_r>0$ such that
$$
	C_r^{-1}\frac{|z-w|}{\rho(z)}\leq \beta (z, w)
	\leq C_r\frac{|z-w|}{\rho(z)},\qquad w\in D^{r}(z)
$$
and
$$
	C_r^{-1}\left(\frac{|z-w|}{\rho(z)}\right)^{\delta}\leq \beta(z, w)
	\leq C_r\left(\frac{|z-w|}{\rho(z)}\right)^{2-\delta},\qquad w \in {\C}\backslash D^{r}(z).
$$
\end{lemma}

\begin{proof}
A completely analogous result was proved for $d_\phi(\cdot, \cdot)$ in~\cite[Lemma~4]{MMO03}. Using \eqref{equi-Bergman-dist}, it can be seen that all remain valid if $d_\phi(z, w)$ is replaced by $\beta(z, w)$.
\end{proof}

For $z\in \C$ and $r>0$, we set $B(z, r)= \{w: \beta(w, z)<r\}$. The estimate in the following technical lemma will be very useful in our analysis.

\begin{lemma}\label{Lma:2.1}  For $p, \epsilon>0$ and $k,l\in\R$, there is some $C>0$ such that
\begin{equation}\label{rho-integral}
	\int_{ {\mathbb C}}\left(\beta(z, w)+1\right)^l  \rho(w)^{k}e^{ - p\left(\frac{|z-w|}{\rho(z)}\right)^{\epsilon} }dA(w)\leq C\rho(z)^{k+2} \qquad\textrm{ for }\,  z\in {\mathbb C},
\end{equation}
and
\begin{equation}\label{rho-integral-m}
	\lim_{R\to \infty} \sup_{z\in \C} \fr 1 {\rho(z)^{k+2}}\int_{ {\mathbb C}\setminus B(z, R)}\left(\beta(z, w)+1\right)^l  \rho(w)^{k}e^{ - p\left(\frac{|z-w|}{\rho(z)}\right)^{\epsilon} }dA(w)=0.
\end{equation}
\end{lemma}

\begin{proof}
Since $ \beta(z, w)+1\ge 1 $, to prove (\ref{rho-integral}) and (\ref{rho-integral-m}), we may assume $l>0$. We claim first that for $z,w\in \C$, we have
\begin{equation}\label{basic-integral}
	\left(\beta(z, w)+1\right)^l \le C  e^{\fr p 2\left(\frac{|z-w|}{\rho(z)}\right)^{\epsilon}}.
\end{equation}
Observe that, for  $w\in \C \backslash D(z)$ (i.e., $|w-z|\geq \rho(z)$),  by Lemma~\ref{distance-d}, we get
\begin{align*}
	\beta(z, w)^l  e^{-\fr p 2\left(\frac{|z-w|}{\rho(z)}\right)^{\epsilon}}
	&\le C \left(\frac{|z-w|}{\rho(z)}\right)^{(2-\delta)l} e^{- \fr p 2\left(\frac{|z-w|}{\rho(z)}\right)^{\epsilon}}\\
	&\le C  \sup_{\left \{\xi : |z-\xi|/\rho(z) \ge 1  \right \}} 
     \left(\frac{|z-\xi|}{\rho(z)}\right)^{(2-\delta)l} e^{- \fr p 2\left(\frac{|z-\xi|}{\rho(z)}\right)^{\epsilon}},
\end{align*}
where the positive constants $C$ are independent of $z, w$. Thus,
\begin{equation}\label{y-estimate}
	\sup_{w\in \C\setminus D(z)} \beta(z, w)^l  e^{-\fr p 2\left(\frac{|z-w|}{\rho(z)}\right)^{\epsilon}}
	\le C \sup_{r\ge 1 }  r^{l(2-\delta)} e^{- \fr p 2 r^\epsilon}= C.
\end{equation}
In addition, it is trivial that $\sup_{w\in   D(z)} \beta(z, w)^l  e^{-\fr p 2\left(\frac{|z-w|}{\rho(z)}\right)^{\epsilon}}\le C$. Therefore,
$$
	\sup_{z, w\in \C}  \left(\beta(z, w)+1\right)^l e^{-\fr p 2\left(\frac{|z-w|}{\rho(z)}\right)^{\epsilon}}   \le  C\sup_{z, w\in \C} \left(\beta(z, w)^l+1 \right) e^{-\fr p 2\left(\frac{|z-w|}{\rho(z)}\right)^{\epsilon}}\\
  <\infty,
$$
which gives (\ref{basic-integral}). Now for all $z\in \C$, (\ref{basic-integral}) and Lemma 2.1 of \cite{HL17} imply
\begin{align*}
	\int_{ {\mathbb C}}\left(\beta(z, w)+1\right)^l  \rho(w)^{k}e^{ - p\left(\frac{|z-w|}{\rho(z)}\right)^{\epsilon} }dA(w)
	&\leq C \int_{ {\mathbb C}}   \rho(w)^{k}e^{ -\fr 12  p\left(\frac{|z-w|}{\rho(z)}\right)^{\epsilon} }dA(w)\\
	&\leq C \rho(z)^{k+2}
\end{align*}
and hence (\ref{rho-integral}) follows.

Choosing $r=1$ in Lemma~\ref{distance-d}, for $\beta(w,z)\geq R$ with $R$ sufficiently large, we see that the following inclusion holds:
$$
   	{\mathbb C}\setminus B(z, R)\subset  {\mathbb C}\setminus D^{(R/{C_1})^{\fr 1{2-\delta}}}(z).
$$
We write $N=N(R)$ for the integer part of $(R/{C_1})^{\fr 1{2-\delta}}$. Then by  (\ref{y-estimate}) and (\ref{eq:1}), we get
\begin{align*}
	&\int_{ {\mathbb C}\setminus B(z, R)}\left(\beta(z, w)+1\right)^l  \rho(w)^{k}e^{ - p\left(\frac{|z-w|}{\rho(z)}\right)^{\epsilon} }dA(w)\\
	& \le C    \int_{{\mathbb C}\setminus D^{(R/{C_1})^{\fr 1{2-\delta}}}(z)}
    \rho(w)^{k}e^{ - \fr p 2 \left(\frac{|z-w|}{\rho(z)}\right)^{\epsilon} }dA(w)\\
	& \le C    \int_{{\mathbb C}\setminus D^{N}(z)}
    \rho(w)^{k}e^{ - \fr p 2 \left(\frac{|z-w|}{\rho(z)}\right)^{\epsilon} }dA(w)\\
	&\le C  \sum_{j=N}^\infty  \int_{ D^{j+1}(z) \backslash D^{j}(z)}
    \rho(w)^{k}e^{ - \fr p 2 \left(\frac{|z-w|}{\rho(z)}\right)^{\epsilon} }dA(w)\\
	& \le C \rho(z)^{k+2} \sum_{j=N}^\infty (j+1)^{|k|\theta +2} e^{-\fr p2 j^\epsilon}.
\end{align*}
This gives  (\ref{rho-integral-m}) because  $\sum_{j=1}^\infty (j+1)^{|k|\theta +2} e^{-\fr p2 j^\epsilon}<\infty$ and $N(R)\to \infty$ as $R\to \infty$.
The proof is completed.
\end{proof}

In 2009, Marzo and  Ortega-Cerd\`{a} \cite{MO09} obtained the following pointwise estimates on the Bergman kernel.

\begin{lemma}\label{basic-est}
{\rm (A)} There exist $C, \epsilon>0$ such that
\begin{equation}\label{eq:02}
|K(w, z)|\leq C\frac{e^{\phi(w)+\phi(z)}}{\rho(w)\rho(z)}e^{ - \left(\frac{|z-w|}{\rho(z)}\right)^{\epsilon} },\qquad w,z\in {\C}.
\end{equation}

{\rm (B)} There exists  some $r_0>0$ such that for $z\in\mathbb{C}$ and  $w\in D^{r_0}(z)$, we have
\begin{equation}\label{eq:03}
	|K(w, z)|\simeq\frac{e^{\phi(w)+\phi(z)}}{\rho(z)^2}.
\end{equation}
\end{lemma}

\begin{proof}
For (A), see Theorem 1.1 of \cite{MO09}, and for (B), see Proposition~2.11 of~\cite{MO09}.
\end{proof}

\begin{corollary}\label {K-integral}  
For $p >0$ and  $l\in\R$, there is some $C>0$ such that
\begin{equation}\label{rho-integral-a}
	\int_\C\left( \beta(z, \xi) +1\right) ^l |K(z, \xi)|^p  e^{-p\phi(\xi)}dA(\xi)\le C \rho(z)^{ 2( 1-p )}e^{p\phi(z)}\, \textrm{ for } \, z\in \C,
\end{equation}
and
\begin{equation}\label{rho-integral-z}
	\sup_{z\in \C}\rho(z)^{ 2( p-1 )}e^{-p\phi(z)}\int_{ {\mathbb
   C}\setminus B(z, R)}\left( \beta(z, \xi) +1\right) ^l |K(z, \xi)|^p  e^{-p\phi(\xi)}dA(\xi)  \to 0
\end{equation}
as $R\to \infty$.
\end{corollary}

\begin{proof}
By Lemma \ref{basic-est} and the estimate (\ref{rho-integral}),
\begin{align*}
	&\int_\C\left( \beta(z, \xi) +1\right) ^l |K(z, \xi)|^p  e^{-p\phi(\xi)}dA(\xi)\\
	&\le C \fr {e^ {p \phi(z) }}{\rho(z)^p}    \int_\C\left( \beta(z, \xi) +1\right) ^l  {\rho(\xi)^{-p}}  e^{ - p \left(\frac{|z-\xi|}{\rho(z)}\right)^{\epsilon} }dA(\xi)\\
	&\le C \rho(z)^{2(1-p)} e^ {p \phi(z) },
\end{align*}
which shows (\ref{rho-integral-a}). The equality in (\ref{rho-integral-z}) can be proved similarly.
\end{proof}

\begin{corollary}\label{L-infty-norm}
For $s\in\R$, there are positive constants $C_1$ and $C_2$ such that
$$
	C_1 \rho(z)^{s-2} e^{2\phi(z)}\le     \| \rho(\cdot)^s K(\cdot, z)\|_{\infty, \phi}  \le C_2 \rho(z)^{s-2} e^{ \phi(z)}
$$
for all $z\in \C$.
\end{corollary}

\begin{proof}
From  (\ref{eq:03}) we see that
\begin{equation}\label{infty-norm-a}
	\| \rho(\cdot)^s K(\cdot, z)\|_{\infty, \phi}  \ge \rho(z)^s K(z, z) \ge C  \rho(z)^{s-2} e^{ 2\phi(z)}.
\end{equation}
To prove the reverse inequality, for $\xi\in D^m(z)\setminus D^{m-1}(z)$ with some $m=1, 2, \ldots$, by \eqref{eq:1} we have
$$
      \rho(\xi)^s   = \rho(z)^s\left(\fr  {\rho(\xi)} {\rho(z)}\right)  ^s  \le C \rho(z)^s(m^\gamma) ^{|s|}.
$$
Notice that $\{ m^{\gamma |s-1|} e^{- (m-1)^\epsilon} \}_{m=1}^\infty$ is bounded, and so
\begin{align*}
	\rho(\xi)^s | K(\xi, z)| e^{-\phi(\xi)}
	&\le C \fr {e^{\phi(z)}}{\rho(z)} \rho(\xi)^{s-1} e^{-\left( \fr {|\xi-z|}{\rho(z)} \right)^\epsilon}\\
	&\le C   \rho(z)^{s-2}{e^{\phi(z)}} m^{\gamma |s-1|} e^{- (m-1)^\epsilon}\\
	&\le C   \rho(z)^{s-2} {e^{\phi(z)}}
\end{align*}
This implies
$$
	\| \rho(\cdot)^s K(\cdot, z)\|_{\infty, \phi}  \le C \rho(z)^{s-2} e^{ \phi(z)},
$$
which completes the proof.
\end{proof}

In our analysis, we need  some auxiliary function spaces. Given a Lebesgue measurable function $f$ on $\C$, we set
\begin{equation}\label{e:local oscillation1}
	\omega (f)(z)=\sup\left\{|f(z)-f(w)|:  \beta(z, w)<r \right\}.
\end{equation}
When defining the following function spaces, we suppress $r$, because the spaces are independent of it, and often choose $r=1$. We denote by $BO$ the class  of all  continuous functions $f$ on ${\C}$ for which $\omega (f)$ is bounded on  ${\C}$ and let $VO$ denote the class of all  continuous functions $f$ for which
$$
	\lim\limits_{z\rightarrow\infty}\omega (f)(z)=0.
$$
For $f\in BO$, set $\|f\|_{BO }=\sup\limits_{z\in {\C}}\omega (f)(z)$. With Lemma \ref{distance-d}, for a function $f$ continuous on $\C$, it is easy to verify that  $f\in BO $ if and only if
\begin{equation}\label{e:local oscillation2}
	\omega_{r}^*(f)(z)=\sup\left\{|f(z)-f(w)|:  w\in D^{r}(z)\right\} \in L^\infty
\end{equation}
with the equivalent semi-norm $ \sup_{z\in \C} \omega_{r}^*(f)(z)$, and $f\in VO $ if and only if $\lim_{z\to \infty} \omega_{r}^*(f)(z)=0$.

Given a Lebesgue measurable function $f$, write
$$
	\widehat{f} (z)=\frac{1}{A\left(D (z)\right)}\int_{D (z)}fdA
$$
for the average of $f$ over $D (z)$. For $0<p<\infty$, we use ${BA} ^p$ to denote the space of all  $f\in L_{loc}^p({\C})$ such that
$$
	\|f\|_{{BA} ^p}=\sup\limits_{z\in {\C}}\left[\widehat{\left(|f|^p\right)} (z)\right]^{\frac{1}{p}}<\infty.
$$
The space ${VA} ^p$ consists of all  $f\in BA^p $ such that
$$
	\lim\limits_{z\rightarrow\infty}\widehat{\left(|f|^p\right)} (z)=0.
$$
Given a Lebesgue measurable function $f$ which is positive or satisfies the condition $f|k_z|^2\in L^1_{2\phi}$ for all $z\in \C$, we define the Berezin transform $\widetilde{f}$ of $f$ by
\begin{equation}\label{berezing}
	\widetilde{f}(z)=\int_{{\C}}f(w)\left|k_{z}(w)\right|^2e^{-2\phi(w)}dA(w).
\end{equation}
For $f\in L_{loc}^p({\C})$, take $d\mu= |f|^p dA$, $p=q$ and $t=2$ in Theorems 3.2 and 3.3 of~\cite{HL17} to obtain
$$
	\sup\limits_{z\in {\C}}\widehat{(|f|^p)} (z)\simeq \sup\limits_{z\in {\C}}\widetilde{(|f|^p)}(z)
$$
and
$$
	\lim\limits_{z\rightarrow\infty}\widehat{(|f|^p)} (z)=0
	\iff \lim\limits_{z\rightarrow\infty}\widetilde{(|f|^p)}(z)=0.
$$
For  $1\leq p<\infty$,  $BMO ^p$ is defined as   the space of all $f\in L_{loc}^p({\C})$ such that
$$
	\|f\|_{BMO ^{p}}=\sup\limits_{z\in {\C}}MO_{p }(f)(z)<\infty,
$$
where
$$
	MO_{p }(f)(z)=\left(\frac{1}{A\left(D (z)\right)}\int_{D (z)}\left|f-\widehat{f} (z)\right|^pdA\right)^{\frac{1}{p}}.
$$
We say $f\in VMO^p$ if  $f\in BMO ^{p}$ and $\lim\limits_{z\rightarrow\infty}MO_{p }(f)(z)=0$. Finally, we observe that
\begin{equation}\label{e:decomposition}
	BMO^p = BO + BO^p\quad{\rm and}\quad VMO^p = VO + VA^p,
\end{equation}
which can be obtained by writing $f = \tilde f + (f-\tilde f)$; for further details, see Theorem~2.5 of~\cite{HL16}.

\section{Hankel operators with $BO$ symbols}
In this section, we show that for $0<p<\infty$, the Hankel operator $H_f: F^p_\phi \to L^p_\phi$ is bounded if $f\in BO$ and compact if $f\in VO$. In addition to the intrinsic interest, these properties will be needed for our characterization of Fredholm Toeplitz operators.

We start with an auxiliary operator. For $l\in\R$, with   kernel $(\beta(\cdot, \cdot)+1)^l |K(\cdot, \cdot)|$ we define an integral operator $G_l$ by
$$
	G_l(f)(z)= \int_{\C} f(\xi) (\beta(z, \xi)+1)^l |K(z, \xi)| e^{-2\phi(\xi)}dA(\xi).
$$
To show that $G_l$ is bounded from $F^p_\phi$ to $L^p_\phi$ also for $0<p<1$, we need the following generalization of a result in~\cite{MMO03}.

\begin{lemma}\label{Lemma 19}
Let $0<p<\infty$. For any $r>0$, there is a $C>0$ such that
\begin{equation}\label{subharmonicity}
	\left |f(z)e^{-\phi(z)}\right|^ p \le C  \fr 1{
  A(D^r(z))}\int_{D^r(z) } \left |f(\xi)e^{-\phi(\xi)}\right|^ p dA(\xi)
\end{equation}
for all $f\in H(\C)$ and $z\in \C$.
\end{lemma}

\begin{proof}
When $1\le p<\infty$, estimate (\ref{subharmonicity}) is just Lemma 19 (a) of \cite{MMO03}. To deal with the other values of $p$, notice that the function $f e^{-H_z}$ with $H_z$ defined as in \cite{MMO03} is holomorphic in $D^r(z)$. Then $\left|f e^{-H_z}\right|^p$ is subharmonic for $0<p<1$ so that the sub-mean value estimate can be applied as in the case $1\le p<\infty$. This shows that the proof of Lemma 19 (a) in \cite{MMO03} works for all $0<p<\infty$.
\end{proof}

\begin{lemma} \label{extened-Bergman-projection}
For  $1\le p\le \infty$, the operator $G_l$ is bounded on $L^p_\phi$; and for  $0< p\le \infty$, $G_l$  is bounded from $F^p_\phi$ to $L^p_\phi$.
\end{lemma}

\begin{proof}
We begin with the case $1\le p\le \infty$. For $f$ Lebesgue measurable, it follows from Lemmas~\ref{Lma:2.1} and~\ref{basic-est} that
 \begin{align*}
       \|G_lf\|_{L^1_\phi} &\le  \int_{\C} e^{-\phi(z)}\, dA(z)
          \int_{\C} |f(\xi)|  (\beta(z, \xi) + 1)^l |K(z, \xi)| e^{-2\phi(\xi)}\, dA(\xi)\\
  &=\int_{\C} |f(\xi)| e^{-2\phi(\xi)}\, dA(\xi) \int_{\C} e^{-\phi(z)}
          (\beta(z, \xi)+1)^l  |K(z, \xi)|\, dA(z)\\
 &\le  C  \int_{\C} |f(\xi)| e^{-\phi(\xi)}\, dA(\xi) \int_{\C}
         (\beta(z, \xi)+1)^l  \frac{1}{\rho(\xi)\rho(z)}e^{ - \left(\frac{|z-\xi|}{\rho(z)}\right)^{\epsilon} }   dA(z)\\
  &\le  C \|f\|_{L^1_\phi}.
\end{align*}
Similarly, we have
\begin{align*}
	\|G_lf\|_{L^\infty_\phi} &\le \sup_{z\in \C} e^{-\phi(z)}
          \int_{\C} |f(\xi)|  (\beta(z, \xi)+1)^l  |K(z, \xi)| e^{-2\phi(\xi)}dA(\xi)\\
	&\le \|G_lf\|_{L^\infty_\phi} \sup_{z\in \C} e^{-\phi(z)}
          \int_{\C}   (\beta(z, \xi)+1)^l  |K(z, \xi)| e^{-\phi(\xi)}dA(\xi)\\
	&\le \|G_l f\|_{L^\infty_\phi}.
\end{align*}
By interpolation (see Theorem~3.5 of \cite{OP15}), we know that  $G_l$ is bounded on $L^p_\phi$ when $1\le p\le \infty$.

Next we treat the case $0<p<1$. By Lemma \ref{distance-d}, we get $\beta(w, \xi)\le   C_1 $ for $\xi\in D(w)$. Then, for $z, w\in \C$ and $\xi, \zeta \in D^2(w)$, we have
$$
	\beta(z, \xi)\le \beta(z, \zeta)+\beta(\zeta, w)+\beta(w, \xi) \le \beta(z, \zeta) +2C_2
$$
and so
\begin{equation}\label{d-estimate}
   \sup_{\xi\in D(w)}\beta(z, \xi)\le (2C_2+1) \inf_{\xi\in D^2 (w)} \left( \beta(z, \xi) +1 \right)
\end{equation}
for all $z, w\in \C$. Therefore, for an $r$-lattice $\{z_j\}_{j=1}^\infty$ and $f\in H(\C)$, using Lemma~\ref{Lemma 19}, we get
\begin{align*}
     &|G_l(f)(z)|^p \le \left( \sum_{j=1}^\infty \int_{D (z_j)}  |f(\xi)|
        \beta(z, \xi)  |K(z, \xi)|  e^{-2\phi(\xi)} dA(\xi)\right)^p \\
 &\le \sum_{j=1}^\infty \left(  \int_{D (z_j)}  |f(\xi)|
        \beta(z, \xi)  |K(z, \xi)|  e^{-2\phi(\xi)} dA(\xi)\right)^p\\
 &\le C\sum_{j=1}^\infty  \sup_{\xi\in D (z_j)} \left(  |f(\xi)|
        \beta(z, \xi)  |K(z, \xi)|  e^{-2\phi(\xi)} \right)^p \rho(z_j)^{2p}\\
 &\le C   \sup_{\xi\in D (z_j)}\int_{D^{2 }(z_j)}
        \left(  |f(\xi)|
        \left( \beta(z, \xi)+1\right)  |K(z, \xi)|  e^{-2\phi(\xi)} \right)^p
         \rho(\xi)^{2p-2}dA(\xi)\\
 &\le C   \int_{\C}
        \left(  |f(\xi)|
       \left( \beta(z, \xi)+1\right)  |K(z, \xi)|  e^{-2\phi(\xi)} \right)^p \rho(\xi)^{2p-2}dA(\xi).
\end{align*}
Thus, for $f\in H(\C)$ and $0<p<1$,
\begin{equation}\label{extened-Bergman-projection-1}
	|G_l(f)(z)|^p\le C \int_\C \left( |f(\xi)| \left( \beta(z, \xi)+1\right) |K(z, \xi)|  e^{-2\phi(\xi)}\right)^p \rho(\xi)^{2p-2} dA(\xi).
\end{equation}
Applying \eqref{rho-integral-a} and Fubini's theorem, we get
\begin{align*}
  \|G_l(f)\|^p_{L^p_\phi}
  &\le C\int_\C e^{-p\phi(z)}dA(z) \int_{\C}
        \left(  |f(\xi)|
        \left(\beta(z, \xi)+1\right)  |K(z, \xi)|  e^{-2\phi(\xi)} \right)^p\\ 
        &\qquad\qquad\qquad\qquad\qquad\qquad\qquad\times \rho(\xi)^{2p-2}dA(\xi)\\
  &=C \int_{\C}
        \left(  |f(\xi)|
          e^{-2\phi(\xi)} \right)^p \rho(\xi)^{2p-2}dA(\xi) \int_\C \left(\beta(z, \xi)+1\right) ^p\\	&\qquad\qquad\qquad\qquad\qquad\qquad\qquad\times |K(z, \xi)|^p  e^{-p\phi(z)}dA(z) \\
 &\le C \int_{\C}
          |f(\xi)|^p
          e^{-p\phi(\xi)}  dA(\xi),
\end{align*}
which completes the proof.
\end{proof}

Relative to the Bergman projection $P$, there is another operator $P_+$, known as the absolute Bergman projection, which has been studied in the Bergman space case. In the context of Fock spaces, we define $P_+ : L^p_\phi \to F^p_\phi$ analogously by setting
$$
	P_+f(z)= \int_{\Cn} f(\xi)|K(z, \xi)| e^{-2\phi(\xi)} dv(\xi), \quad z\in \C.
$$
The following boundedness properties of $P_+$ follow directly from the preceding lemma.

\begin{corollary}\label {alsolute-projection}
The operator $P_+$ is bounded on $L^p_\phi$ for $1\le p\le \infty$, and it is bounded from $F^p_\phi$ to $L^p_\phi$ for  $0<p<1$.
\end{corollary}

\begin{lemma}\label{compact-supported}
 For each $R>0$ and $\varepsilon>0$, there is a function $h_R$ defined on $\C$ such that $\mathrm{supp}\, h_R$, the support of $h_R$,  is compact and
$$
	h_R|_{B(z, R)}\equiv 1, \verb# #\|h_R\|_{L^\infty}= 1, \verb# # \| h_R\|_{BO}<\varepsilon.
$$
\end{lemma}

\begin{proof} Without loss of generality, we may assume $R>\fr 1 \varepsilon$. Define
$$
	h_R(z)=
	\begin{cases}
		1,     &  \beta(z, 0)<R; \\
		2-\fr {\beta(z, 0)} R,   \verb# #& R\le \beta(z, 0)< 2R; \\
		0,  & \beta(z, 0)\ge 2R.
	\end{cases}
$$
We show that this function satisfies the required properties.

Clearly, $\mathrm{supp}\, h_R$ is compact and  $h_R |_{B(z, R)}\equiv 1$.  For $z$ with $ \beta(z, 0)\le R-1$ or $ \beta(z, 0)\ge 2R+1$, we see that $\omega(h_R)(z)=0$. For $z$ with $R+1\le \beta(z, 0)\le 2R-1$  and  $w\in B(z, 1)$, we have
$$
	|h_R(w)-h_R(z)|=\left| \fr {\beta(w, 0) }R - \fr {\beta(z, 0)} R \right| \le  \fr {\beta(w, z)} R  < \fr 1 R.
$$
Now for $z$ with $R-1\le \beta(z, 0)< R+1$, when $R\geq 1$, and $w\in B(z, 1)$, we have
$$
	|h_R(w)-h_R(z)| \le 1- \left(2-\fr {R+1} R\right) = \fr 1R.
$$
Similarly, for $z$ with $2R-1\le \beta(z, 0)< 2R+1$, when $R\geq 1$, and $w\in B(z, 1)$, we have
$$
	|h_R(w)-h_R(z)| \leq \left(2-\fr {2R-1} R -0 \right) = \fr 1R.
$$
It follows from these estimates that we have $\omega(h_R)(z)\le \fr 1 R <\varepsilon$.
\end{proof}

The following lemma is well known in other weighted Fock spaces but requires extra work when the Laplacian of the weight function is a doubling measure.

\begin{lemma} \label{hankel-1}
 Let $0<p<\infty$. If $f\in L^\infty(\C)$ has compact support, then the Hankel operator $H_f$ is compact from $F^p_\phi$ to $L^p _\phi$.
\end{lemma}

\begin{proof}
We denote by  $B(F^p_\phi)$  the unit ball of $F^p_\phi$. To show that $H_f$ is compact from $F^p_\phi$ to $L^p _\phi$, we only need to show that $H_f(B(F^p_\phi))$  is relatively compact in $L^p_\phi$. By Lemma~\ref{Lemma 19}, there is some constant $C$ such that, for $f\in B(F^p_\phi)$ and $z\in \C$,
$$
	|f(z)e^{-\phi(z)}|^p\le C \rho(z)^{-2} \|f\|_{p,\phi}^p\le C \rho(z)^{-2}.
$$
This implies that $B(F^p_\phi)$ are uniformly bounded on any compact subset of $\C$. Thus, $B(F^p_\phi)$ is a normal family. Hence, to prove  $H_f(B(F^p_\phi))$  is relatively compact in $L^p_\phi$, it suffices to prove that
\begin{equation}\label{comp-supp-a}
	\lim_{j\to \infty} \|H_f(g_j)\|_{p, \phi}= \lim_{j\to \infty}\| fg_j-P (f g_j) \|_{p, \phi}=0
\end{equation}
for any bounded sequence $\{g_j\}_{j=1}^\infty$ in $F^p_\phi$ converging to 0 uniformly on any compact subset of  $\C$. But this will be an easy consequence of the limits
\begin{equation}\label{comp-supp-c}
	\lim_{j\to \infty} \|f g_j \|_{p, \phi}= 0
\end{equation}
and
\begin{equation}\label{comp-supp-b}
	\lim_{j\to \infty} \|P(f g_j )\|_{p, \phi}= 0.
\end{equation}

The equality(\ref {comp-supp-c}) is trivial for $f\in L^\infty$ with compact support and $\{g_j\}_{j=1}^\infty$ bounded in $F^p_\phi$ converging to 0 uniformly  on compact subsets. As a simple consequence, observe that (\ref{comp-supp-c}) implies  (\ref{comp-supp-b}) for $1\le p<\infty$ because $P$ is bounded on $L^p_\phi$.

It remains to prove (\ref{comp-supp-b}) for $0< p<1$. Without loss of generality, we may assume that the support of $f$ is contained in some $D(0, \sigma)$. Write $d\mu =|f| dA$. As stated on page~869 of~\cite{MMO03}, there are some $C>0$ and $0<s<1$ such that $\rho(z)\le C|z|^s$ for $|z|>1$. Hence
\begin{equation} \label{growth-of-rho}
	\lim_{|z|\to \infty} |z|-\rho(z) =\infty,
\end{equation}
and so there is an $R>0$ so that $\widehat{\mu}(\xi) = \int_{D( \xi)}d\mu/ {A(D(\xi)) }=0$ when $|\xi| \ge R$. Then applying Lemma 2.4 of~\cite{HL17} gives the following estimate:
\begin{align*}
	| P(f g_j)(z) | &\le \int_{\C} | f(\xi)  g_j(\xi) K(z, \xi)|e^{-2\phi(\xi)} dA(\xi)\\
	&= \int_{\C} |   g_j(\xi) K(\xi, z)|e^{-2\phi(\xi)} d\mu(\xi)\\
	&\le C \int_{\C} |   g_j(\xi) K(\xi, z)|e^{-2\phi(\xi)}  \widehat{\mu} (\xi)dA(\xi)\\
	&\le |f\|_{L^\infty} \int_{D(0, R)} |   g_j(\xi) K(\xi, z)|e^{-2\phi(\xi)}   dA(\xi).
\end{align*}
Now, with the same approach as that of (\ref{extened-Bergman-projection-1}) we obtain
$$
	|P(f g_j)(z) |^p \le C \|f\|_{L^\infty}^p \int_{D(0, 2R)} |   g_j(\xi) K(\xi,
        z)|^p e^{-2p \phi(\xi)} \rho(\xi)^{2p-2}  dA(\xi)
$$
where we have  chosen $R$ large enough so  that $|\xi|+ \rho(\xi) \le  2R$ for $|\xi |\le R$. Therefore, by Corollary~\ref{K-integral}, we get
\begin{align*}
	\| P(f g_j)\|_{p, \phi}^p 
	&\le C \|f\|_{L^\infty} ^p\int_{\C} e^{-p \phi(z)} dA(z)  \int_{D(0, 2R)} |g_j(\xi) K(\xi,
        z)|^p \\
        &\qquad\qquad\qquad\qquad\qquad\qquad\qquad\times e^{-2p \phi(\xi)} \rho(\xi)^{2p-2}  dA(\xi)\\
	&=  C\|f\|_{L^\infty}^p \int_{D(0, 2R)} |   g_j(\xi)|^p\rho(\xi)^{2p-2} e^{-2p
   \phi(\xi)}  dA(\xi)\int_{\C}| K(\xi, z)|^p\\ 
	&\qquad\qquad\qquad\qquad\qquad\qquad\qquad\times e^{-p \phi(z)} dA(z)\\
	&\le C \|f\|_{L^\infty}^p \int_{D(0, 2R)} |g_j(\xi)^p  e^{-p \phi(\xi)}  dA(\xi)\to 0
\end{align*}
as $j\to \infty$, which completes the proof.
\end{proof}

The next theorem provides useful properties of Hankel operators with $BO$ (and $VO$) symbols and it plays an important role in the study of Fredholmness of Toeplitz operators. While we give the proof for all values of $p$, we note that for $1\le p<\infty$, the result is also a consequence of Theorem~3.3 of~\cite{HL16}.

\begin{theorem} \label{hankel} 
Suppose that $0<p<\infty$.
\begin{itemize}
\item[(i)] If $f\in BO$, then   $H_f$ is bounded from $F^p_\phi$ to $L^p _\phi$ and we have the following norm estimate:
\begin{equation}\label{hankel-norm}
   \| H_f \|_{F^p_\phi \to L^p_\phi }\le C \|f\|_{BO}.
 \end{equation}

\item[(ii)]   If $f\in VO$, then   $H_f$ is compact from $F^p_\phi$ to $L^p _\phi$.
\end{itemize}
\end{theorem}

\begin{proof}
For $f\in BO$, it is easy to verify that
\begin{equation}\label{bo-function}
	|f(z)-f(\xi)|\le \|f\|_{BO}(\beta(z, \xi)+1).
\end{equation}
As shown in the proof of Theorem~3.2 of~\cite{LZZ19}, we have
\begin{equation}\label{bergman-proj}
	P(g)=g\, \textrm{ for }\, g\in F^p_\phi \, \textrm{ with }\, 0<p\le \infty.
\end{equation}
Then
\begin{align*}
	| H_f(g)(z)| 
	&\le \int_{\C} | f(\xi)-f(z) ||g(\xi)||K(z, \xi)| e^{-2\phi(\xi)} dv(\xi)\\
	& \le C\|f\|_{BO} \int_{\C} \left (\beta(z, \xi) +1\right) |g(\xi)||K(z, \xi)| e^{-2\phi(\xi)} dv(\xi).
\end{align*}
Therefore, Lemma \ref{extened-Bergman-projection} implies that,  for $0<p<\infty$, $H_f$ is bounded from $F^p_\phi$ to $L^p_\phi$ with the norm estimate \eqref{hankel-norm}.

Now we  suppose $f\in VO$.  For $\varepsilon>0$, fix $r>0$ so that $\omega(f)(w)<\varepsilon$ whenever $\beta(w, 0)\ge r$. Then, for $w$ with $\beta(w, 0)>r$, we have some $\eta(w)$, $\beta(\eta, 0) =r$ so that $\beta(w, 0)=\beta(w, \eta)+\beta(\eta, 0)$. Hence, by
$$
	|f(w)-f(0)|\le \|f\|_{BO}(\beta(\eta, 0)+1) +\varepsilon (\beta(w, \eta)+1)
$$
there is an $R>r$   so that, when $\beta(w, 0)>R$,
$$
	\fr {|f(w)|}{\beta(w, 0)} \le \fr {\|f\|_{BO}  (r+1) +|f(0)| } {\beta(w, 0)}+\varepsilon \fr{r+1} {\beta(w, 0)} <2\varepsilon.
$$
Set $f_R= f h_R $ with $h_R$ to be as in Lemma \ref{compact-supported}. For such $f_R$, it is easy to see that $\omega(f-f_R)(z)=0$ when $\beta(z, 0)<R-1$, and  $\omega(f-f_R)(z)=\omega(f)(z)$ when $\beta(z, 0)>2R+1$. For those $z$ that satisfy $R-1\le \beta(z, 0)\le 2R+1$ and $w\in \beta(z, 1)$, we have
\begin{align*}
 	&\left|(f(z)-f_R(z))-(f(w)-f_R(w))\right|\\
 	&\leq |f(w)|\left|h_R(z)-h_R(w)\right|+(1-h_R(z))\left|f(w)-f(z)\right| \\
	&\leq |f(w)|\omega (h_R)(z)+\omega (f)(z) \leq  |f(w)|\frac{1}{R }+\omega (f)(z)\\
	&=\frac{|f(w)|}{\beta(w, 0) }\frac{\beta(w, 0) }{R }+\omega (f)(z) \leq  6\varepsilon+\varepsilon.
\end{align*}
Therefore, by (\ref{hankel}),
\begin{equation}\label {compact-approximation}
	\| H_f-H_{f_R}   \|_{F^p_\phi \to L^p_\phi}\le C \| f-f_R  \|_{BO}\le C \varepsilon,
\end{equation}
where the constant $C$ is independent of  $\epsilon$. Notice that $H_{f_R} $ is compact, and so is $H_f $  since the family of all compact linear operators from $F^p_\phi$ to $L^p_\phi$ is closed under the operator norm. The proof is completed.
\end{proof}

\begin{remark}\label{bo-continuity}
A careful check of  the preceding proof shows that the continuity of $f \in BO$ (or $VO$) is not needed for the conclusion of Theorem~\ref{hankel}.
\end{remark}

\section{Fredholm theory}\label{Fredholm theory}\label{Fredholm section}
A linear mapping $T$ on a topological vector space $X$ is said to be Fredholm if
$$
	\dim \ker T < \infty\ {\rm and}\ \dim X/T(X) < \infty.
$$
When $X$ is a Banach space, it is well known that $T$ is Fredholm if and only if $T+K(X)$ is invertible in the Calkin algebra $B(X)/K(X)$, where $B(X)$ and $K(X)$ stand for the spaces of bounded and compact operators, respectively. From this, it follows that an operator on a Banach space is Fredholm if and only if there are bounded operators $A$ and $B$ on $X$ such that
$$
	AT = I + K_1\ {\rm and}\ TB = I + K_2
$$
for some compact operators $K_1$ and $K_2$ acting on $X$. Because two Toeplitz operators often commute modulo compact operators, the previous characterization for their Fredholmness is almost tailor-made for large classes of symbols.

These characterizations of Fredholm operators are not true in general if $X$ is not a Banach space. However, an adequate theory can still be developed for quasi-Banach spaces under some additional conditions, which is important in certain PDE problems; see, e.g.~\cite{MR1443193}. A pair $(X, \|\cdot\|)$ is said to be a quasi-Banach space if $\|\cdot\|$ satisfies all the properties of a norm except for the triangle inequality and if there is a constant $C>0$ such that
$$
	\|x+y\| \leq C (\|x\| + \|y\|)
$$
for all $x,y\in X$. Observe that all generalized Fock spaces $F^p_\phi$ are quasi-Banach spaces. We now define an additional property for quasi-Banach spaces as in~\cite{MR1419319}.

\begin{definition}\label{quasi-Banach}
A quasi-Banach space $X$ is said to be dual rich if for all nonzero vectors $x\in X$, there is a continuous linear functional $x^*$ such that $x^*(x) = 1$.
\end{definition}

As an example, we mention that every Banach space is dual rich, and so are $\ell^p$ with $0<p<1$, while none of the $L^p(\C^n, dv)$ spaces with $0<p<1$ is dual rich; see~\cite{MR1419319}.

For the Fock space $F^p_\phi$ with a small exponent, we have the following lemma which  is an easy consequence of Theorem 5.1 of~\cite{LZZ19}.

\begin{lemma}\label{dual rich Fock}
If $0<p<1$, then the  Fock space $F^p_\phi$ is a dual rich quasi-Banach space.
\end{lemma}

The following result is needed  when we characterize Fredholm operators on $F^p_\phi$ for $0<p<1$.

\begin{theorem}\label{dual rich Fredholm}
A bounded linear operator on a dual rich quasi-Banach space $X$ is Fredholm if and only if it has a regularizer; that is, there exists a bounded linear operator $S$ on $X$ such that $ST-I$ and $TS-I$ are both compact on $X$.
\end{theorem}
\begin{proof}
See Section 3.5.1 of~\cite{MR1419319}.
\end{proof}

Before we can embark on the proof of our main theorem, we need a few preliminary results that illustrate the role played by the Berezin transform and the normalized reproducing kernel function.

\begin{lemma} \label{LeBMO1}  Let  $f\in VO$. Then
\begin{equation}\label{vo-z}
	\lim_{z\to \infty}  \left(f- \widetilde{f}\, \right)  (z)=0.
\end{equation}
\end{lemma}

\begin{proof} We start by applying Lemmas~\ref{Lma:2.1} and~\ref{basic-est} to get
\begin{multline*}
	\sup_{z\in \C} \int_{\C  \setminus B(z,R)} (\beta(z, \xi)+1) |k_z(\xi)|^2 e^{-2 \phi(\xi)} dv(\xi) \\
	\le C \sup_{z\in \C} \int_{\C  \setminus B(z,R)} (\beta(z, \xi)+1)  \rho^{-2}(\xi)   e^{ - 2 \left(\frac{|z-\xi|}{\rho(z)}\right)^{\epsilon} }dA(\xi) \to 0
\end{multline*}
as $R\to \infty$. Thus,
for each $\varepsilon>0$, there is an $R>0$ such that
\begin{equation}\label{vo-c}
	\int_{\C \setminus B(z,R)} (\beta(z, \xi)+1) |k_z(\xi)|^2 e^{-2 \phi(\xi)} dv(\xi)<\varepsilon
\end{equation}
for all $z\in \C$. Since $f\in VO$,  there is some $\rho>0$ such that
$$
	\sup_{\xi\in B(z,R)} |f(\xi)- f(z)|<\varepsilon
$$
whenever $|z|>\rho$. Notice also that $\int_{\C} |k_z(\xi)|^2 e^{-2 \phi(\xi)} dv(\xi)=1$.
Thus, for  $|z|>\rho$,
\begin{align*}
 \left|f- \widetilde{f}\, \right|  (z)&\le  \int_{\C} \left|f(z)- f
      (\xi)\right||k_z(\xi)|^2 e^{-2 \phi(\xi)} dv(\xi)\\
 &\le \left\{ \int_{B(z, R)} +   \int_{\C \setminus B(z, R)}\right\}
      \left|f(z)- f (\xi)\right||k_z(\xi)|^2 e^{-2 \phi(\xi)} dv(\xi)\\
 &\le \varepsilon+  \|f\|_{BO} \int_{\C \setminus  B(z, R)}
       \left(\beta(\xi, z)+1 \right)|k_z(\xi)|^2 e^{-2 \phi(\xi)} dv(\xi)\\
  &\le\left( 1+    \|f\|_{BO} \right)  \varepsilon,
\end{align*}
which gives (\ref{vo-z}) and hence the proof is complete.
\end{proof}

\begin{lemma}\label{fredholm-lamma}
Let $0<p<\infty$ and let $f\in VO$. If $z_j\in \C$,
$$
	\lim_{j\to \infty} z_j=\infty, \verb#  # \lim_{j\to \infty} f( z_j)=0,
$$
then
\begin{equation}\label{vmo-1}
	\lim_{k\to \infty} \|T_{f}(k_{z_j, p})\|_{p, \phi} =0.
\end{equation}
\end{lemma}

\begin{proof}
We treat the case  $0<p\le 1$ first.
Suppose $f\in VO$.
For $\varepsilon>0$, Corollary  \ref{K-integral} gives some $R>1$ such that
\begin{equation}\label{vo-a}
	\int_{\C \setminus B(z, R)} ( \beta(z, \xi)+1)^p \rho(\xi)^{ 2p-2} |k_{z, p}( \xi)|^p     e^{- p \phi(\xi)} dA(\xi)
	\le \left(\fr {\varepsilon}{2 \| f \|_{BO}+1}\right)^p
\end{equation}
for all $z\in \C$. Furthermore, for the fixed $\varepsilon$ and $R$, we  have some $\rho>0$ so that
$$
	\sup_{\xi\in B(z, R)} | f  (\xi)-  {f }(z)|<\varepsilon
$$
whenever  $|z|>\rho$. Notice also that $(f(\cdot)-f(z_j))k_{z_j, p}(\cdot) K(\cdot, z)\in H(\C)$. With a similar approach as that of proving (\ref{extened-Bergman-projection-1}), we have
\begin{multline*}
	\left\{ \int_{\C} \left |{f }(\xi)-  {f }(z_j) \right| |k_{z_j, p }(\xi)||K(\xi, z)| e^{-2\phi(\xi)} dA(\xi) \right\}^p\\
	\le C \int_{\C} \left | {f }(\xi)-  {f }(z_j) \right|^p |k_{z_j, p }(\xi)|^p |K(\xi, z)|^p  e^{-2p\phi(\xi)}\rho(\xi)^{ 2p-2} dA(\xi).
\end{multline*}
Then for $|z_j|>\rho$, by Corollary \ref{K-integral}, we get
\begin{align*}
	&\int_{\C} \left( \int_{\C} \left |{f }(\xi)-  {f }(z_j) \right| |k_{z_j, p }(\xi)||K(\xi, z)| e^{-2\phi(\xi)} dA(\xi) \right)^p e^{-p\phi(z)} dA(z)\\
	&\leq C \int_{\C} \left | {f }(\xi)-  {f }(z_j) \right|^p |k_{z_j, p }(\xi)|^p e^{-2p\phi(\xi)}\\
	&\qquad\qquad\times \rho(\xi)^{ 2p-2} dA(\xi)\int_{\C} |K(\xi, z)|^p e^{-p\phi(z)} dA(z)\\
	&=C \int_{\C} \left | {f }(\xi)-  {f }(z_j) \right|^p |k_{z_j, p }(\xi)|^p e^{-p\phi(\xi)} dA(\xi)\\
	&\le C \Bigg( \varepsilon^p \left \|  k_{z_j, p }\right\|_{p, \phi}^p\\  
	&\qquad\qquad + \|f\|_{BO}^p \int_{\C\setminus B(z_j, R)} (\beta(\xi, z)+1)^p |k_{z_j, p }(\xi)|^p  e^{-p\phi(\xi)} dA(\xi) \Bigg )\\
   &\le C \varepsilon^p.
\end{align*}
The constants $C$ above are independent of $\varepsilon$. This, together with  Lemma \ref{Lma:2.1}, Corollary \ref{alsolute-projection} and the obvious inequality
$$
	\left|  T_{f}(k_{z_j, p })(z) \right|
	\le  \int_{\C} \Big(|f(\xi)- f(z_j)| +  |f(z_j)| \Big) |k_{z_j, p }(\xi)|K(\xi, z)| e^{-2\phi(\xi)} dv(\xi),
$$
implies
$$
	\limsup_{j\to \infty} \left \| T_f(k_{z_j, p })  \right\|_{p, \phi}^p 
	\le C\varepsilon^p + \limsup_{j\to \infty}  |f(z_j)|^p \left\| P_+(|k_{z_j, p }|\right\|_{p, \phi}^p = C \varepsilon^p,
$$
which gives \eqref{vmo-1} and completes the proof.
\end{proof}

\begin{lemma}\label{compact-berezin}
If $0<p<\infty$ and  $T$ is  compact linear operator on $F^p_\phi$, then
\begin{equation}\label{berezin-trans}
	\lim_{z\to \infty} \widetilde{T}(z)=0,
\end{equation}
where $\tilde T$ is the Berezin transform of the operator $T$ defined by $\widetilde{T}(z)= \langle Tk_z, k_z \rangle$ for $z\in \C$ and $\langle \cdot, \cdot \rangle$ is the inner product on $F^2_\phi$.
\end{lemma}

\begin{proof}
When $1\le p<\infty$, the conclusion (\ref{berezin-trans}) is trivial because $F^p_\phi$ is a Banach space. Our proof here works well for all $0<p<\infty$.
As stated on page~476  of~\cite{HL17} or alternatively using the approach of Theorem~7 in~\cite{DFG2002}, a subset $E \subset F^p_\phi$ is relatively compact in $F^p_\phi$ if and only if for each $\varepsilon>0$ there is some $R>0$ such that
$$
    \sup_{f\in E} \int_{|z|\ge R} \left |f(z)e^{-\phi(z)}\right|^p dA(z)<\varepsilon.
$$
Since $T$ is compact on $F^p_\phi$, we have
 \begin{equation}\label{eq31}
\lim_{R\to \infty} \sup_{f\in F^p_\phi, \|f\|_{p, \phi}\le 1} \int_{|z|\ge R} \left |Tf(z)e^{-\phi(z)}\right|^p dA(z)=0.
    \end{equation}
Hence, by  the reproducing formula for functions in $F^p_\phi$, given in~\eqref{bergman-proj}, we have
$$
	\widetilde{T}(z)=\fr {\|K_z\|_{p, \phi}}{K(z, z)} \langle T(k_{z, p})(\cdot), K(\cdot, z)\rangle=  \fr {\|K_z\|_{p, \phi}}{K(z, z)}T(k_{z, p})(z).
$$
This and \eqref{eq31} imply
\begin{align*}
	| \widetilde{T}(z)|^p &\le C \left|\rho(z)^{\fr 2p} T(k_{z, p})(z) e^{-\phi(z)}\right|^p \\
	&\le C \int_{B(z, 1)} \left|T(k_{z})(\xi) e^{-\phi(\xi)}\right|^p dA(\xi) \to 0
\end{align*}
as $z\to \infty$. This completes the proof.
\end{proof}

\begin{theorem} \label{cor-1}
Let $0<p<\infty$ and  $f\in VO$. Then
\begin{itemize}
\item[(i)]  $T_{f-\widetilde{f}}$ is compact on $F^p_\phi$.

\item[(ii)] The Toeplitz operator $T_{f}$ is compact on $F^p_\phi$ if and only if $\widetilde{f}(z)\to 0$ as $z\to \infty$.
\end{itemize}
\end{theorem}

\begin{proof}
(i) First  we show  that the Berezin transform $\widetilde{f}$ is continuous on $\C$.  By Lemma~\ref{basic-est}, for $w\in B(z, 1)$,
\begin{align*}
	\left| f(\xi ) \left|k_w(\xi)\right|^2e^{-2\phi(\xi)}\right|
	&\le C \Big(|f(\xi)-f(z)|+|f(z)|\Big)  \rho(w)^{-2} e^{ - 2
     \left(\frac{|\xi-w|}{\rho(w)}\right)^{\epsilon} }\\
	&\le C \Big(\|f\|_{BO} +|f(z)|\Big) \rho(z)^{-2} e^{ -
     \left(\frac{|\xi-z|}{\rho(z)}\right)^{\epsilon} },
\end{align*}
and $\int_{\C}  e^{ - \left(\frac{|\xi-z|}{\rho(z)}\right)^{\epsilon} }dA(\xi)<\infty$. Therefore, by Lebesgue's dominated convergence theorem, it follows that
$$
	\lim_{w\to z} \int _\C f(\xi ) \left|k_w(\xi)\right|^2e^{-2\phi(\xi)} dA(\xi)=
    \int _\C f(\xi ) \left|k_z(\xi)\right|^2e^{-2\phi(\xi)}dA(\xi).
$$
Hence $f-\widetilde{f}\in C(\C)$,  which  together with Lemma \ref{LeBMO1} implies that $ f- \widetilde{f} \in L^\infty\cap VO$.

That $T_{f-\widetilde{f}}$ is compact on $F^p_\phi$ is an easy consequence of  the stronger assertion that $T_g$ is compact if $g\in L^\infty$ and $\lim_{z\to \infty} g(z)=0$. Indeed, to verify this, write $\chi_R$ for the characteristic function of $D(0, R)$.  Theorem~3.2 of~\cite{HL17} tells us that $T_{ | g |\chi_R}$ is compact on $F^p_\phi$, and so is $ T_{g \chi_R}$. Now, since $\|g-g\chi_R\|_{L^\infty} \to 0$ as $r\to \infty$, we have
$$
	\|T_{g\chi_R}-T_g\|_{F^p_\phi\to F^p_\phi} \le C \|g-g \chi_R\|_{L^\infty}\to 0
$$
as $R\to 0$, and hence $T_g$ is compact on $F^p_\phi$.

(ii) Suppose that $T_f$ is compact. Then $\lim_{z\to \infty} \widetilde{T_f}(z)=0$ by Lemma \ref{compact-berezin}. Notice that, for $f\in VO \subseteq BO$, from \eqref{bo-function} and Corollary \ref{K-integral} we obtain
 $$
   \int_\C |k_z(\xi)| e^{-2\phi(\xi)} dA(\xi) \int_\C |f(\zeta)
       k_z(\zeta)K(\xi, \zeta))| e^{-2\phi(\zeta)} dA(\zeta)<\infty.
 $$
 Then, applying  Fubini's theorem,  we get
\begin{equation}\label{berezin-on-T}
\begin{split}
	\widetilde{T_f}(z) &=   \langle T_{f}k_z, k_z  \rangle\\	
	&= \int_\C  f(\zeta) k_z(\zeta)  e^{-2\phi(\zeta)} dA(\zeta) \int_\C \overline{ k_z(\xi)}  K(\xi, \zeta)  e^{-2\phi(\xi)} dA(\xi)= \widetilde{f}(z).
\end{split}
\end{equation}
Therefore, we have $\lim_{z\to \infty} \widetilde{f}(z)=0$. The converse follows from the simple identity
$$
	T_f= T_{\widetilde{f}}+ T_{f-\widetilde{f}},
$$
and the proof is complete.
\end{proof}

Before proving our main result, we comment briefly on the previous attempt to prove it in the Hilbert space setting.

\begin{remark}\label{remark on AHV}
In Proposition~4.1 of~\cite{AHV}, the following statement was given. \textit{Let $f:\C\to\C$ be a continuous function in $A$, where $A = L^\infty \cap VO$ or $A=L^\infty\cap VMO^1$. Then $f$ is bounded away from zero on $\C\setminus {D}(0,R)$, for some $R>0$, if and only if there is a continuous function $g\in A$ such that $f(z)g(z)\to 1$ as $z\to \infty$.}

There are two problems with this statement. First, it is only known in the setting of the standard Fock spaces when the space $VO$ is defined using the Euclidean metric (see~Proposition~9 of~\cite{AV18} and Lemma~17 of~\cite{BC87}); that is, when the Bergman metric is replaced by the Euclidean metric in~\eqref{e:local oscillation1} (or, equivalently, when $D^r(z)$ is replaced by the Euclidean disk $D(z,r)$ in~\eqref{e:local oscillation2}). Second, the claim that $g$ is continuous does not seem correct in general when $f$ is only in $L^\infty\cap VMO^1$. In~\cite{AHV}, the above statement was then used to construct a regularizer of $T_f$ as follows:
$$
	T_f T_g = I + T_{fg-1} - PM_f H_g.
$$
The proof of the following theorem indicates how these problems can be avoided.
\end{remark}

We are now ready to state and prove our main result.

 \begin{theorem} \label{Fredholm}
 Let $f\in VMO^1$ and $0<p<\infty$.  Then the Toeplitz operator $T_f$ is Fredholm on $F^p_\phi$ if and only if
\begin{equation}\label{Fredhol-1}
	0<\liminf_{|z|\to \infty} |\widetilde{f}(z)|\le \limsup_{|z|\to \infty} |\widetilde{f}(z)|< \infty .
\end{equation}
\end{theorem}

\begin{proof}  According to~\eqref{e:decomposition}, we have $VMO^1  = VO+VA^1 $, that is, $f\in VMO^1 $ if and only if there are functions $f_1\in VO$ and $f_2\in VA ^1$ such that
\begin{equation}\label{vmo-a}
  f=f_1+f_2.
\end{equation}
Set $\mu= |f_2|\, dA$. Then
$$
	\lim_{z\to \infty} \fr 1 {|D (z)|} \int_{D (z)} d\mu =0,
$$
which means that $\mu$ is a vanishing Fock-Carleson measure. Thus, Theorem 3.3 of \cite{HL17} (see also Theorem 4.1 of  \cite{OP15} for the case $1\le p<\infty$) yields
\begin{equation} \label{Fredhol-a}
	\lim_{z\to \infty} \widetilde{|f_2|}(z) = 0.
\end{equation}
Then,  $T_{f_2}$ is compact on $F^p_\phi$ for all possible $0<p<\infty$. Consequently, $T_f$ is Fredholm if and only if $T_{f_1}$ is Fredholm. Furthermore, \eqref{Fredhol-a} implies
\begin{equation}\label{vmo-b}
	\liminf_{|z|\to \infty} |\widetilde{f}(z)|  =\liminf_{|z|\to \infty} |\widetilde{f_1}(z)|\, \textrm{ and }\,   \limsup_{|z|\to \infty} |\widetilde{f}(z)|  =\limsup_{|z|\to \infty} |\widetilde{f_1}(z)|.
\end{equation}
Therefore, we need only to prove the desired conclusion for symbols in $VO$.

Now we suppose $f \in VO$ and that   $T_{f }$ is Fredholm on $F^p_\phi$. Then,  $T_{f }$ is  bounded on $F^p_\phi$. We treat the case  $0< p\le 1$ first.
By  Lemma~\ref{basic-est}  and Corollary \ref{L-infty-norm}, we have
$$
      \left \|\rho^{2-\fr 2p}(\cdot) \left(\rho^{\fr 2p -1}(z)\,   k_{z}(\cdot) \right )\right \|_{\infty, \phi}\le C
$$
for all $z\in \C$. Notice that
$$
	\left |\langle T_{f} k_z, k_z\rangle\right| \simeq\left |\langle T_{f} k_{z,p},  \rho^{\fr 2p -1}(z) \,  k_{z}\rangle\right|.
$$
Then by \eqref{berezin-on-T}, as shown in the proof of Theorem 5.1 from \cite{LZZ19},
$$
	|\widetilde{f}(z)|
	\le C \| T_{f} k_{z,p}\|_{p, \phi} \, \left \|\rho^{2-\fr 2p}(\cdot) \left(\rho^{\fr 2p -1}(z)\,   k_{z}(\cdot) \right )\right \|_{\infty, \phi}
	\le C \|T_{f}\|_{F^p_\phi\to F^p_\phi}.
$$
Hence,
\begin{equation}\label{vmo-j}
	\limsup_{z\to \infty}  |\widetilde{f}(z)|<\infty.
\end{equation}
Now for $1< p<\infty$, using \eqref{berezin-on-T} and H\"older's inequality, we have $|\widetilde{f}(z)|\le C \|T_{f}\|_{F^p_\phi\to F^p_\phi}$.

If $\liminf_{z\to \infty}  |\widetilde{f}(z)|>0$ were not true, we would have  some sequence $\{z_j\}_{j=1}^\infty$ in $\C$ such that $\lim_{j\to \infty} z_j=\infty$ and $ \lim_{j\to \infty} \widetilde{f}(z_j)=0$. According to Lemma\ref{fredholm-lamma},
$$
	\lim_{j\to \infty} \|T_{\widetilde{f}}(k_{z_j, p})\|_{p, \phi} =0,
$$
which implies that
$$
	\lim_{j\to \infty} \left \|\left( S T_{\widetilde{f}}\right) k_{z_j, p} \right \|_{p, \phi} =0
$$
for any bounded operator $S$ on $F^p_\phi$. Notice that
$$
	\left| \left(S T_{\widetilde{f}}\right)^{\widetilde{ } }(z_j)\right| 
	\simeq \left |\left \langle\left( S T_{\widetilde{f}}\right) k_{z_j ,p} ,  \rho^{\fr 2p -1} (z_j)\,  k_{z_j}\right \rangle\right|  \ \ \textrm{ if } \ 0<p\le 1,
$$
and
$$
	\left| \left(S T_{\widetilde{f}}\right)^{\widetilde{ } }(z_j)\right|
	\simeq \left |\left \langle\left( S T_{\widetilde{f}}\right) k_{z_j ,p} ,    k_{z_j, p'}\right \rangle\right|   \ \ \textrm{ if } \ 1<p< \infty,
$$
where $p'$ is the conjugate exponent of $p$. Thus,
\begin{equation}\label{vmo-d}
	\lim_{j\to \infty}\left| \left(S T_{\widetilde{f}}\right)^{\widetilde{ } }(z_j)\right|=0
\end{equation}
On the other hand, by Theorem~\ref{cor-1}, we know that $T_{\widetilde{f}}= T_{f}  + T_{\widetilde{f}-f} $ is also Fredholm on $F^p_\phi$, and so we can apply Lemma ~\ref{dual rich Fock} and Theorem~\ref{dual rich Fredholm} to get some bounded linear operator $S$ on $F^p_\phi$ such that
\begin{equation*}\label{vmo-c}
	S T_{\widetilde{f}}=I +K,
\end{equation*}
where $I$ is the identity operator and $K$ is some compact operator on $F^p_\phi$. By Lemma \ref{compact-berezin}, we have $\lim_{z\to \infty}\widetilde{K}(z)=0$ and hence
$$
	\lim_{j\to \infty} \left| {\left( S T_{\widetilde{f}}\right)}^{ \widetilde{  }}({z_j})\right| 
	\ge 1 - \lim_{j\to \infty}  \left|\widetilde{K}({z_j})\right| =1,
$$
which contradicts \eqref{vmo-d}. This completes the proof of the necessary condition.

Conversely, suppose that $f\in VO$ and $\widetilde{f}$ satisfies \eqref{Fredhol-1}. Then there are positive constants $R$, $c$ and $C$ such that
\begin{equation}\label{vmo-z}
	c\le |\widetilde{f}(z)|\le C
\end{equation}
for $|z|\ge R$. On the other hand, it follows easily from Lemma \ref{LeBMO1} that $\widetilde{f}\in VO\cap L^\infty$. We define a function $g$ on $\C$ by
$$
	g(z)=
	\begin{cases}
	\ 0,\verb#         # |z|<R; \\
	\fr  1 {\ \widetilde{f}\ }, \verb#        # |z|\ge R.
\end{cases}
$$
Then
$$
	g\in  L^\infty \quad \textrm{and} \quad \omega(g)(z) \le \fr 1{c^2} \omega(\widetilde{f})(z)\to 0
$$
as $z\to \infty$. Using Theorem~\ref{hankel} and Remark \ref{bo-continuity}, we see  that $H_g$ is compact from $F^p_\phi$ to $L^p_\phi$. Furthermore,
$$
	\widetilde{f}(z)g(z)=
	\begin{cases}
	0,\verb#        # |z|<R; \\
	1 , \verb#        # |z|\ge R.
\end{cases}
$$
Therefore, $T_{\widetilde{f}g }=I- T_{\chi_R}$ on $F^p_\phi$. Thus,
\begin{align*}
	T_{\widetilde{f}} T_g &= PM_{\widetilde{f}} PM_g\\
	&=PM_{\widetilde{f}} \left[I-(I-P)\right]M_g\\
	&= T_{\widetilde{f}g}-PM_{\widetilde{f}} H_g\\
	&= I - T_{\chi_R} -PM_{\widetilde{f}} H_g= I+K_1,
\end{align*}
where $K_1= - T_{\chi_R}-PM_{\tilde f} H_g$ is compact on $F^p_\phi$ because $PM_{\tilde f}$ is clearly bounded from $L^p_\phi$ to $F^p_\phi$. Similarly, $ T_g T_{\widetilde{f}}= I+ K_2$ for some compact operator $K_2$   on $F^p_\phi$. We conclude that $T_{\widetilde{f}}$ is Fredholm and the proof is complete.
\end{proof}

\begin{corollary}\label{ess spect}
Let $0<p<\infty$ and $f\in VMO^1$. If~\eqref{Fredhol-1} holds, then
\begin{equation}\label{e:ess spect}
	\sigma_{\rm ess}(T_f) = \bigcap_{R>0} \overline{\tilde f(\C^n\setminus B(0,R))}
\end{equation}
and the essential spectrum $\sigma_{\rm ess}(T_f)$ is connected.
\end{corollary}
\begin{proof}
The previous theorem gives the description in~\eqref{e:ess spect} and the connectedness follows from this because $\tilde f$ is continuous.
\end{proof}

\section{Open problems and conclusions}\label{open problems and conclusions}

Our work completes the study of Fredholm Toeplitz operators on the most general weighted Fock spaces $F^p_\phi$ where their boundedness and compactness are well understood but only in dimension one. As our work makes use of a number of results, for example, on the reproducing kernel and the Bergman distance that are only known in dimension one, it seems that the generalization to the higher dimensions is an onerous task. However, we still think that the ideas presented in our current work may well serve as a blueprint for such generalizations and we conjecture that our main result remains true for doubling Fock spaces over $\C^n$. It is worth noting that pointwise estimates for the weighted Bergman kernel in several complex variables was recently obtained in~\cite{D15}, which may serve as (partial) substitutes for the estimates used in our present work.

Perhaps another suitable starting point for further study is the notion of Fock spaces $\mathcal A^2(\Psi)$ in $\C^n$ defined with logarithmic growth functions $\Psi$; for further details about these types of Fock spaces, see~\cite{SY13}. In particular, we note that both the spaces $F^p_\phi$ and $\mathcal A^p(\Psi)$ serve as generalizations of the classical  Fock space that have potential to stimulate further interest in function-theoretic operator theory and interplay between several branches of analysis.


\begin{thebibliography}{99}
\bibitem{AHV} Al-Qabani, A., Hilberdink, T., Virtanen, J. A.: Fredholm theory of Toeplitz operators on doubling Fock Hilbert spaces. Math. Scand. 126, no. 3, 593--602 (2020)

\bibitem{AV18} Al-Qabani, A., Virtanen, J. A.: Fredholm theory of Toeplitz operators on standard weighted Fock spaces. Ann. Acad. Sci. Fenn. Math. 43, 769--783  (2018)

\bibitem{BC87} Berger, C. A., Coburn, L. A.: Toeplitz operators on the Segal-Bargmann space. Trans. Amer. Math. Soc. 301, 813--829 (1987)

\bibitem{BLY2021} Borichev, A., Le, V. A., Youssfi, E. H.: On the dimension of the Fock type spaces. J. Math. Anal. Appl. 503, no. 1, Paper No. 125288, 14 pp. (2021)

\bibitem{CCK12} Cho, H. R., Choe, B. R., Koo, H.: Fock-Sobolev spaces of fractional order. Potential Anal. 43, 199--240 (2015)

\bibitem{CZ12} Cho, H. R., Zhu, K.: Fock-Sobolev spaces and their Carleson measures. J. Funct. Anal. 263, 2483--2506 (2012)

\bibitem{Ch91}  Christ, M.: On the $\overline{\partial}$ equation in weighted $L^2$ norms in $\C$. J. Geom. Anal. 1, 193--230 (1991)

\bibitem{CO11} Constantin, O., Ortega-Cerd\`a J.: Some spectral properties of the canonical solution operator to $\overline \partial$ on weighted Fock spaces. J. Math. Anal. Appl. 377, 353--361 (2011)

\bibitem{CP15} Constantin, O., Pel\'{a}ez, J.: Integral operators, embedding theorems and a Littlewood-Paley formula on weighted Fock spaces. J. Geom. Anal. 26, 1109--1154 (2016)

\bibitem{D15} Dall'Ara, G. M.: Pointwise estimates of weighted Bergman kernels in several complex variables. Adv. Math. 285, 1706--1740 (2015)

\bibitem{DFG2002} D\"{o}rfler, M., Feichtinger, H. G., Gr\"{o}chenig, K.: Compactness criteria in function spaces. Colloq. Math. 94, no. 1, 37--50 (2002)

\bibitem{FH17} Fulsche, R., Hagger, R.: Fredholmness of Toeplitz Operators on the Fock Space. Complex Anal. Oper. Theory 13, 375--403 (2019)

\bibitem{JPR87} Janson, S., Peetre, J., Rochberg, R.: Hankel forms and the Fock
space, Rev. Mat. Iberoamericana 3, 61--138 (1987)

\bibitem{HS19} Hagger, R., Seifert, C.: Limit operators techniques on general metric measure spaces of bounded geometry. J. Math. Anal. Appl. 489, no. 2, 124180, 36 pp. (2020)

\bibitem{HL16} Hu, Z. J., Lv, X. F.: Hankel operators on weighted Fock spaces. SCIENTIA SINICA Mathematica 46, 141--156 (2016)

\bibitem{HL17} Hu, Z. J., Lv, X. F.: Positive Toeplitz operators between different doubling Fock spaces, Taiwanese J. Math. 21, 467--487 (2017)

\bibitem{HLS19} Hu, Z. J., Lv, X. F., Schuster, A. P.: Bergman spaces with exponential weights. J. Funct. Anal. 276, 1402--1429 (2019)

 \bibitem{HV19} Hu, Z., Virtanen, J. A.: Fredholm Toeplitz operators with $VMO$ symbols and the duality of generalized Fock spaces with small exponents. Proc. Roy. Soc. Edinburgh Sect. A 150, no. 6, 3163--3186 (2020)

\bibitem{MR1443193} Kalton, N., Mitrea, M.: Stability results on interpolation scales of quasi-Banach spaces and applications. Trans. Amer. Math. Soc. 350, 3903--3922 (1998)

\bibitem{Krantz1} Krantz, S. G.: Function theory of several complex variables. Reprint of the 1992 edition. AMS Chelsea Publishing, Providence, RI (2001)

\bibitem{Krantz2} Krantz, S. G.: Geometric analysis of the Bergman kernel and metric. Graduate Texts in Mathematics, 268. Springer, New York (2013)

\bibitem{LZZ19} Lou, Z., Zhu, K., Zhuo, Z.: Atomic decomposition and duality for a class of Fock spaces. Complex Var. Elliptic Equ. 64, 1905--1931 (2019)

\bibitem{MMO03} Marco, N., Massaneda, X., Ortega-Cerd\`{a}, J.: Interpolating and sampling sequences for entire functions. Geom. Funct. Anal. 13, 862--914 (2003)

\bibitem {MO09} Marzo, J., Ortega-Cerd\`{a}, J.: Pointwise estimates for the Bergman kernel
of the weighted Fock space. J. Geom. Anal. 19, 890--910 (2009)

\bibitem{Ma95} Mattila, P.: Geometry of Sets and Measures in Euclidean Spaces. Fractals
and Rectifiability. Cambridge University Press, Cambridge (1995)

\bibitem{OP15} Oliver, R., Pascuas, D.: Toeplitz operators on doubling Fock spaces. J. Math. Anal. Appl. 435, 1426--1457 (2016)

\bibitem{P-W1990} Pasternak-Winiarski, Z.: On the dependence of the reproducing kernel on the weight of integration. J. Funct. Anal. 94 , no. 1, 110--134 (1990)

\bibitem{MR1419319} Runst, T., Sickel, W.: Sobolev spaces of fractional order, Nemytskij operators, and nonlinear partial differential equations. Walter de Gruyter \& Co., Berlin (1996)

\bibitem{MR1137884} K. Stroethoff, Hankel and Toeplitz operators on the Fock space. Michigan Math. J. 39, no. 1, 3--16  (1992)

\bibitem{SS09} Schneider, G., Schneider, K.: Generalized Hankel operators on the Fock space. Math. Nachr. 282, 1811--1826 (2009)

\bibitem {SS11} Schneider, G., Schneider, K.: Generalized Hankel operators on the Fock space II. Math. Nachr. 284, 1967--1984 (2011)

\bibitem{SY13} Seip, K., Youssfi, El H.: Hankel operators on Fock spaces and related Bergman kernel estimates, J. Geom. Anal. 23, 170-201 (2013)

\bibitem{Tu05} Tung, J.: Fock Spaces. Ph.D. thesis, University of Michigan (2005)

\bibitem{Zh12} Zhu, K.:  Analysis on Fock spaces. Springer, New York (2012)
\end{thebibliography}
\end{document}